\documentclass[a4paper]{amsart}

\usepackage{latexsym, amssymb,amsmath, amsthm,amsfonts,color}
\newcommand{\kk}{\Bbbk}
\newcommand{\kv}{{\kk[\mathcal{V}]}}

\newcommand{\kvg}{{\kk[\mathcal{V}]^{G}}}

\newcommand{\spec}{\mathrm{Spec}}
\newcommand{\vv}{\mathcal{V}}

\newcommand{\M}{{\mathcal{M}}}

\newcommand{\bA}{{\mathbf{A}}}
\newcommand{\bZ}{{\mathbf{Z}}}
\newcommand{\ba}{{\mathbf{a}}}
\newcommand{\bb}{{\mathbf{b}}}
\newcommand{\bc}{{\mathbf{c}}}
\newcommand{\bbeta}{{\boldsymbol{\beta}}}
\newcommand{\bgamma}{{\boldsymbol{\gamma}}}

\newcommand{\cc}{\mathcal{C}}

\newcommand{\ww}{\mathcal{W}}

\def\SL{\operatorname{SL}}

\def\GL{\operatorname{GL}}
\def\Ga{{\mathbb G}_{a}}
\def\Gm{{\mathbb G}_{m}}

\def\Trdeg{\operatorname{trdeg}}

\def\chr{\operatorname{char}}
\def\Quot{\operatorname{Quot}}

\def\sgn{\operatorname{sgn}}

\def\Z{\mathbb{Z}}
\def\N{\mathbb{N}}
\def\C{\mathbb{C}}

\def\codim{\operatorname{codim}}

\def\Tr{\operatorname{Tr}}

\def\rk{\operatorname{rk}}
\newtheorem{Lemma}{Lemma}[section]

\newtheorem{thm}[Lemma]{Theorem}
\newtheorem{cor}[Lemma]{Corollary}
\newtheorem{prop}[Lemma]{Proposition}

\theoremstyle{definition}
  \newtheorem{Def}[Lemma]{Definition}  %needs a capital as def already defined

\theoremstyle{remark}
  \newtheorem{rem}[Lemma]{Remark}
\newtheorem{eg}[Lemma]{Example}

\newtheoremstyle{Acknowledgments}% name
  {}% {\topsep}%      Space above
    {}% {\topsep}%      Space below
     {}%         Body font
     {}%         Indent amount (empty = no indent, \parindent = para indent)
    {\bfseries}% Thm head font
    {}%        Punctuation after thm head
     {.5em}%     Space after thm head: " " = normal interword space;
     {\thmname{#1}\thmnumber{ }\thmnote{ (#3)}}%         Thm head spec (can be left empty, meaning `normal')
\theoremstyle{Acknowledgments}
\newtheorem{ack}{Acknowledgments.}

\title{The separating variety for $2 \times 2$ matrix invariants}

\author{Jonathan Elmer}
\address{Middlesex University\\
The Burroughs, Hendon, London\\
NW4 4BT UK}
\email{j.elmer@mdx.ac.uk}

\date{\today}

\begin{document}

\maketitle

\begin{abstract}
Let $G$ be a linear algebraic group acting linearly on a vector space (or more generally, an affine variety) $\vv$, and let $\kvg$ be the corresponding algebra of invariant polynomial functions. A separating set $S \subseteq \kvg$ is a set of polynomials with the property that for all $v,w \in \vv$, if there exists $f \in \kvg$ separating $v$ and $w$, then there exists $f \in S$ separating $v$ and $w$.

In this article we consider the action of $G = \GL_2(\C)$ on the $\C$-vector space $\M_2^n$ of $n$-tuples of $2 \times 2$ matrices by simultaneous conjugation. Minimal generating sets $S_n$ of $\C[\M_2^n]^G$ are well-known, and $|S_n| = \frac16(n^3+11n)$. In recent work, Kaygorodov, Lopatin and Popov \cite{Lopatin2x2} showed that for all $n \geq 1$, $S_n$ is a minimal separating set by inclusion, i.e. that no proper subset of $S_n$ is a separating set. This does not necessarily mean that $S_n$ has minimum cardinality among all separating sets for $\C[\M_2^n]^G$. Our main result shows that any separating set for $\C[\M_2^n]^G$ has cardinality $\geq 5n-5$. In particular, there is no separating set of size $\dim(\C[\M_2^n]^G) = 4n-3$ for $n \geq 3$. Further, $S_3$ has indeed minimum cardinality as a separating set, but for $n \geq 4$ there may exist a smaller separating set than $S_n$. We show that a smaller separating set does in fact exist for all $n \geq 5$. We also prove similar results for the left-right action of $\SL_2(\C) \times \SL_2(\C)$ on $\M_2^n$.
\end{abstract}

\section{Introduction}
\subsection{Matrix Invariants}
Let $\kk$ be an infinite field and let $\M_d$ denote the space of $d \times d$ matrices with coefficients in $\kk$. The linear algebraic group $G: = \GL_d(\kk)$ acts on $\M_d$ by conjugation. More generally we can consider the action of $G$ on $\M_d^n$ by simultaneous conjugation.

The elements of $\M_d^n$ can be viewed as $n$-tuples $\bA = (A_1,A_2, \ldots, A_n)$, or as $d \times d$ matrices with entries in $\kk^n$. We call these $n$-matrices for short. For $g \in G$ we write the conjugation action as
\[g \cdot \bA:= (gA_1g^{-1}, \ldots, gA_ng^{-1}).\] The question of determining whether a pair of matrices lie in the same $G$-orbit is a staple of undergraduate linear algebra. On the other hand, the question of determining whether a pair of $n$-matrices lie in the same $G$-orbit for $n \geq 2$ is the archetypal ``wild'' problem, see for example \cite{Belitskii}. 

There is an action of $\GL_n(\kk)$ on $\M_d^n$ which commutes with the conjugation action: namely, for $h \in \GL_n(\kk)$ and an $n$-matrix $\bA = ({\bf a}_{ij}) \in \M_d(\kk^n)$ we write $h \star \bA$ for the $n$-matrix whose $i,j$ entry is
\begin{equation}\label{commuting action} (h \star \bA)_{ij} = h({\bf a}_{ij}). \end{equation}

Now for $1 \leq i,j \leq d$ and $1 \leq k \leq n$, let $x_{ij}^{(k)}$ denote the linear functional $\M_d^n \rightarrow \kk$ which picks out the $i,j$th entry of of $A_k$, and introduce generic matrices
\[X_k:= \begin{pmatrix} x_{11}^{(k)} & x_{12}^{(k)} \ldots & x_{1d}^{(k)} \\ x_{21}^{(k)} & x_{22}^{(k)} \ldots & x_{2d}^{(k)} \\
\vdots & \vdots & \vdots \\ x_{d1}^{(k)} & x_{d2}^{(k)} \ldots & x_{dd}^{(k)} 
 \end{pmatrix}.\]
Then we have \[\kk[\M_d^n] = \kk[x_{ij}^{(k)}: i,j = 1, \ldots, d, k = 1, \ldots, n].\]

The action of $G$ on $\M_d^n$ induces an action of $G$ on $\kk[\M_d^n]$ by algebra automorphisms: we define
\[(g \cdot f)(\bA) = f(g^{-1} \cdot \bA)\] for all $g \in G$, $f \in \kk[\M_d^n]$ and $\bA \in \M_d^n$. The set $\kk[\M_d^n]^G$ of fixed points of this action forms a $\kk$-subalgebra. Elements of $\kk[\M_d^n]^G$ are called matrix invariants. The algebra $\C[\M_d^n]^G$ has been intensely studied over the years. A minimal generating set is known for arbitrary $n$ only in the cases $d \leq 2$. We ignore $d=1$ as in this case $G$ acts trivially. For $d=2$ we have the following result \cite{LeBruynProcesi}:

\begin{prop}\label{lbp}[LeBruyn-Procesi] The following set $S_n$ of invariants minimally generates $\C[\M^n_2]^G$ as an algebra:
\begin{itemize}
\item $\Tr(X_i), i=1, \ldots, n$.
\item $\det(X_i), i = 1,\ldots n$.
\item $\Tr(X_iX_j), 1 \leq i<j \leq n$.
\item $\Tr(X_iX_jX_k), 1 \leq i<j<k \leq n$.
\end{itemize}
\end{prop}
 
\subsection{Separating Invariants}

Now consider a more general situation in which a linear algebraic group $G$ defined over $\kk$ acts linearly on an affine $\kk$-variety $\vv$. Let $\kv$ denote the algebra of polynomial functions on $\vv$. Then $G$ acts on $\kvg$ according to the formula
\begin{equation}\label{action} (g \cdot f) (v) = f(g^{-1} \cdot v).\end{equation} We denote by $\kvg$ the subalgebra of $\kv$ fixed by this action. For $v, w \in \vv$ and any $f \in \kvg$ we have in $f(v) = f(w)$ if $v \in Gw$, but the converse is not true in general.  

\begin{eg}\label{nonsep}
 Let $G = \GL_2(\C)$ as in Section 1.1. It's well known that
\[\C[\M^1_2]^G = \C[\Tr(X_1),\det(X_1)].\]
However, the matrices
\[A = \begin{pmatrix} 0 & 1 \\ 0 & 0 \end{pmatrix}, A' = \begin{pmatrix} 0 & 0 \\ 0 & 0 \end{pmatrix}\]
do not lie in the same orbit ($G$ fixes $A'$) and we have $\Tr(A) =\Tr(A') = 0$ and $\det(A) = \det(A') = 0$.
\end{eg}

If $f \in \kvg$ and $f(v) \neq f(w)$ we say that $f$ {\it separates} $v$ and $w$. We say that $v$ and $w$ are {\it separated by invariants} if there exists an invariant separating $v$ and $w$.  In case $G$ is reductive, we have that $f(v) = f(w)$ for all $f \in \kvg$ if and only if $\overline{Gv} \cap \overline{Gw} \neq \emptyset$ where the bar denotes closure in the Zariski  topology, see  \cite[Corollary~6.1]{Dolgachev} (note that in Example \ref{nonsep} above we have $A' \in \overline{GA}$ and $G$ is reductive). In particular, the invariants separate the orbits if $G$ is a finite group.

One can in principle separate orbits whenever one can find an explicit generating set for $\kvg$, but this is an extremely difficult problem in general. For this reason, Derksen and Kemper introduced the following in 2002 \cite[Definition~2.3.8]{DerksenKemper}:

\begin{Def}\label{sepdef}
Let $S \subseteq \kvg$. We say $S$ is a {\it separating set} for $\kvg$ if the following holds for all $v, w \in \vv$:
\begin{equation*} s(v) = s(w) \ \text{for all} \ s \in S \Leftrightarrow f(v) = f(w)  \ \text{for all} \ f \in \kvg. \end{equation*}
\end{Def}
Separating sets of invariants have been an area of much recent interest. In general they have nicer properties and are easier to construct than generating sets. For example, if $G$ is a finite group acting on a vector space $V$, then the set of invariants of degree $\leq |G|$ is a separating set \cite[Theorem~3.9.14]{DerksenKemper}. This is also true for generating invariants if $\chr(\kk) = 0$ \cite{FleischmannNoetherBound}, \cite{FogartyNoetherBound} but fails for generating invariants in the modular case. Separating sets for the rings of invariants $\kk[V]^{C_p}$, where $\kk$ is a field of characteristic $p$ and $C_p$ the cyclic group of order $p$ and $V$ is indecomposable were constructed in \cite{SezerCyclic}. Corresponding sets of generating invariants are known only when $\dim(V) \leq 10$ \cite{WehlauCyclicViaClassical}. For the (non-reductive) linear algebraic group $\Ga$ of a field of characteristic zero, separating sets for $\kk[V]^{\Ga}$ for arbitary indecomposable linear representations $V$ were constructed in \cite{ElmerKohls}. These results were extended to decomposable representations in \cite{DufresneElmerSezer}. Even for indecomposable representations, generating sets are known only where $\dim(V) \leq 8$ \cite{Bedratyuk7}.  Finally, for an arbitrary (i.e. non-linear) $\Ga$-variety $\vv$, the algebra of invariants $\kk[\vv]^{\Ga}$ may not be finitely generated, but it is known that there must exist a finite separating set \cite{KemperSeparating} and finite separating sets have been constructed for many examples where $\kk[\vv]^{\Ga}$ is infinitely generated \cite{DufresneKohlsFiniteSep, DufresneKohlsSepVar}. 

Let $S$ be a separating set for $\kvg$ consisting of homogeneous polynomials. The subalgebra $\kk[S]$ of $\kvg$ generated by $S$ is called a {\it separating algebra}. By \cite[Proposition~3.2.3]{DufresnePhD}, the quotient fields of $\kk[S]$ and $\kvg$ have the same transcendence degree over $\kk$. Then by \cite[Proposition~2.3(b)]{Giral} we get that $\dim(\kk[S]) = \dim(\kvg)$. Consequently, the size of a separating set is bounded below by the dimension of $\kvg$. A separating set whose size equals the dimension of $\kvg$ is sometimes called a polynomial separating set, because it necessarily generates a polynomial subalgebra of $\kvg$. On the other hand, there always exists a separating set of size $\leq 2\dim(\kvg)+1$, albeit such a separating set may necessarily contain non-homogeneous polynomials; see \cite[Theorem~5.3]{KamkeKemper} for a proof.

\iffalse
If $G$ is reductive, then any such separating set $S$ for $\kvg$ has the property that $\kvg$ is integral over $\kk[S]$ (see \cite[Corollary~4.2]{DufresneSeparating}). 
\fi

%D+K Prop 2.3.10 shows this under additional f.g. assumption.

\subsection{Statement of results}

We return to the notation of Section 1.1. Recently, Kaygorodov, Lopatin and Popov \cite{Lopatin2x2} showed that $S_n$ is a minimal separating set for $\C[\M_2^n]^G$ by inclusion - i.e. no proper subset of $S_n$ is a separating set. (The authors also considered the problem over fields of finite characteristic, but we will not). Note that this does not necessarily mean that $S_n$ has minimal cardinality as a separating set. Our main results are as follows:

\begin{thm}\label{mainthm} Let $n \geq 2$ and suppose that $S \subseteq \C[\M_2^n]^G$ is a separating set. Then $|S| \geq 5n-5$.
\end{thm}

We will see in section two that $\dim(\C[\M_2^n]^G ) = 4n-3$ for $n \geq 2$, and that $\dim(\C[M_2]^G =1)$.   Thus, the previous theorem implies

\begin{cor}\label{maincor} Let $n \geq 3$. Then there does not exist a polynomial separating set for $\C[\M_2^n]^G$.
\end{cor}

The cardinality of $S_n$ is $\frac16(n^3+11n).$  For $n \geq 7$ this exceeds the upper bound $2 \dim(\C[\M_2^n]^G)+1 = 8n-5$. So $S_n$ does not have minimal cardinality for $n \geq 7$. We also prove

\begin{thm}\label{mainthm2} Let $n \geq 3$. Then there exists a separating set, $S'_n$ for $\C[\M_2^n]^G$ with cardinality $\frac12 (n^2 + 9n- 16)$. $S'_n$ consists of homogeneous polynomials.
\end{thm}

The following table compares our lower bound $5n-5$ with the dimension of $\C[\M_2^n]^G$ and the sizes of $S'_n$ and $S_n$ and the known upper bound $8n-5$ for small values of $n$:

\begin{center}
\begin{table}[h]
\begin{tabular}{c|ccccc}
$n$ & $\dim(\C[\M_2^n]^G)$ &$|S'_n|$ & $|S_n| $& Lower bound & Upper bound \\ \hline
%1 & 2 & 2 & \\
2 & 5 & 5 &5 & 5 & 11\\
3 & 9 & 10&10 & 10 & 19 \\
4 & 13 & 18& 18 & 15 & 27\\
5 & 17 & 27& 30 & 20 & 35\\
6 & 21 & 37 & 47 & 25 & 43\\
7 & 25 & 48 & 70 & 30 & 51\\
8 & 29 & 60 & 100 & 35 & 59
\end{tabular}
\end{table}
\end{center}

In particular, our results imply that $S_3$ is indeed a separating set of minimal cardinality, but that $S_4$ may not be so, and $S_n$ for $n \geq 5$ is not. Note however, than our separating set $S'_n$ may not have minimum cardinality for $n \geq 4$ either, and certainly does not have minimum cardinality for $n \geq 8$. 

\subsection{Structure of paper} This paper is organised as follows. In Section 2 we define the separating variety for the action of a linear algebraic group on an affine variety. We explain how the geometry of the separating variety places obstructions on the existence of small separating sets, using ideas found primarily in the work of Dufresne and others \cite{DufresneSeparating, DufresneJeffries,DufresneKohlsSepVar}. In Section 3 we gather the results we need on matrix invariants, and compute a decomposition of the separating variety for $2 \times 2$ matrix invariants into irreducible components. We then compute the dimension of these components and prove our main results on lower bounds. In Section 4 we show how to construct the smaller separating sets $S'_n$. In Section 5 we prove similar results for the algebra of matrix semi-invariants.

\begin{ack} This research was partially funded by the EPSRC small grant scheme, ref: EP/W001624/1. The author thanks the research council for their support. Special thanks go Emilie Dufresne and Fabian Reimers for a number of suggestions and improvements.
\end{ack}

\section{Invariants of algebraic groups and the separating variety}

In this section we return to the situation and notation of Section 1.2. We note that $G$ also acts (by the same formula \eqref{action}) on the field $\kk(\vv)$ of rational functions $\vv \rightarrow \kk$. We denote the field of rational functions invariant under this action by $\kk(\vv)^G$.

The following is a well-known consequence of a theorem of Rosenlicht: see \cite[Corollary of Lemma~2.4]{PopovVinberg} for a modern proof.
\begin{prop}\label{trdeg} In the above setting we have
\[\Trdeg_{\kk}(\kk(\vv)^G) = \dim(\vv) - \dim(G) + \min_{v \in \vv}\dim(G_v)\] where $G_v$ denotes the stabiliser of a point $v \in\vv$.
\end{prop}

The following is also well-known. For lack of a reference we provide a proof:
\begin{prop}\label{italian} Suppose there exists no non-trivial character $G \rightarrow \kk^*$. Then 
\[\Quot(\kvg) = \kk(\vv)^G.\]
\end{prop}

\begin{proof} It is clear that the right-hand side contains the left. So let $\frac{f_1}{f_2} \in \kk(\vv)^G$. We may assume $f_1$ and $f_2$ are coprime in $\kv$. Now for any $g \in G$ we have
\[\frac{f_1}{f_2} = \frac{g \cdot f_1}{g \cdot f_2}.\]
Then since $f_1$ and $f_2$ are coprime and the action of $G$ preserves degree we must have
\[g \cdot f_1 = \lambda_g f_1, g \cdot f_2 = \lambda_g f_2\] for some constant $\lambda_g \in \kk$. Moreover one sees easily that the assignment $g \rightarrow \lambda_g$ is a homomorphism $G \rightarrow \kk^*$. Since there exists no non-trivial character $G\rightarrow \kk^*$ we get that $f_1, f_2 \in \kvg$ as required.
\end{proof}

\begin{cor}\label{dimensionofinvariantring}
Suppose there exists no non-trivial character $G \rightarrow \kk^*$. Then 
\[\dim(\kvg) = \dim(\vv) - \dim(G) + \min_{v \in \vv}\dim(G_v).\] 
\end{cor}
We note that this applies in particular when $G = \SL_2(\C)$.

The main tool in our proofs will be the separating variety. This was introduced by Kemper in \cite{KemperCompRed}:

\begin{Def}\[ \mathcal{S}_{G,\vv} = \{(v, w) \in \vv^2: f(v)= f(w) \ \text{for all} \ f \in \kvg \} .\]
\end{Def}

In other words, the separating variety is the subvariety of $\vv^2$ consisting of pairs of points which are not separated by any invariant. 

We define $I_{G,\vv}$ to be the ideal of $\kk[\vv^2]$ consisting of the polynomial functions which vanish on $\mathcal{S}_{G,\vv}$. Clearly this is a radical ideal. Then a separating set can be characterised as a subset $S \subseteq \kk[V]^G$ which cuts out the separating variety in $\vv^2$, in other words (see \cite[Theorem~2.1]{DufresneSeparating})

\begin{prop}\label{radversion}
$S \subseteq \kk[\vv]^G$ is a separating set if and only if 
\[V_{\vv^2}(\delta(S)) = \mathcal{S}_{G,\vv}.\]
\end{prop}
where  $\delta: \kk[\vv] \rightarrow \kk[\vv^2] = \kk[\vv] \otimes \kk[\vv]$ is defined by
\[\delta(f) = 1 \otimes f - f \otimes 1.\] 

Equivalently, via the Nullstellensatz, $S$ is a separating if and only if
\[\sqrt{(\delta(f): f \in S)} = {I_{G,\vv}}.\]
%radical not needed on LHS
Consequently the size of a separating set for $\vv$ is bounded below by the minimum number of generators of $I_{G,\vv}$ up to radical, that is, the minimum number of elements generating any ideal whose radical is $I_{G,\vv}$ (this is sometimes called the {\it arithmetic rank} of $I_{G,\vv}$). We then find, using Krull's height theorem, (see e.g. \cite[Theorem~10.2]{Eisenbud}) that :
\begin{prop}\label{krullbound}
Let $S \subseteq \kk[\vv]^G$ by a separating set. Then $|S| \geq \codim_{\vv^2}(\cc)$ for all irreducible components $\cc$ of $\mathcal{S}_{G,\vv}$.
\end{prop}

Therefore, in order to use Proposition \ref{krullbound} above to find lower bounds for separating sets, we must decompose $\mathcal{S}_{G,\vv}$ into irreducible components. As a first step, we observe that the separating variety contains the following subvariety, which we call the {\it graph} of the action:
\begin{Def} \[\Gamma_{G,\vv} = \{(v,gv): v \in \vv, g \in G \}.\]
\end{Def}
If $G$ is connected and reductive, then $\overline{\Gamma_{G,\vv}}$ is an irreducible component of $\mathcal{S}_{G,\vv}$. Its dimension is easily seen to equal $\dim(\vv) + \dim(G) - \min\{\dim(G_v): v \in \vv\}$, and we note that in case there is no nontrivial character $G \rightarrow \C^*$ that
\begin{equation}\label{graphdim} \dim\overline{(\Gamma_{G,\vv})} = 2 \dim(\vv) - \dim(\kvg). \end{equation} However,  the separating variety may have extra components and some of these may have smaller dimension. These components are an obstruction to the existence of small separating sets.

Stronger obstructions may be obtained by taking a closer look at the geometry of $\mathcal{S}_{G,\vv}$. Recall that a Noetherian toplogical space $\vv$ is said to be {\it connected in dimension $k$} if the following holds: for each closed subvariety $Z \subseteq \vv$ with dimension $< k$, the complement $\vv \setminus Z$ is connected. If the same holds for all $Z \subseteq \vv$ with $\codim_{\vv}(Z)  > k$, we say that $\vv$ is {\it connected in codimension $k$}. Note that if $\vv$  is equidimensional, or all irreducible components of $\vv$ intersect nontrivially then we have $\dim(Z) = \dim(\vv)- \codim_{\vv}(Z)$; consequently $\vv$ is connected in dimension $k$ if and only if it is connected in codimension $\dim(\vv)-k$.

 Now we recall Grothendieck's connectedness theorem (see \cite[Expos\'e~XIII, Theorem~2.1]{SGA2}): suppose $(R, \mathfrak{m})$ is a complete local ring of dimension $n$ such that $\spec(R)$ is connected in dimension $k<n$, and let $f_1,f_2, \ldots, f_r \in \mathfrak{m}$. Then $\spec(R/(f_1,f_2, \ldots, f_r))$ is connected in dimension $k-r$.  

The idea is to apply this to $\vv^2 = \spec(\kk[\vv^2])$. Of course, $\kk[\vv^2]$ is graded-local but not a complete ring, but we can bring the connectedness back from that of the completion $\spec(\widehat{\kk[\vv^2]})$ using some ideas of Reimers \cite{reimers} to obtain the following:

\begin{prop}\label{grothbound} Suppose $\mathcal{S}_{G,\vv}$ is not connected in codimension $k$, and let $S \subseteq \kvg$ be a separating set. Suppose further that all irreducible components of $\mathcal{S}_{G,\vv}$ intersect nontrivially, and that there does not exist a non-trivial character $G \rightarrow \kk^*$. Then $|S| \geq \dim(\kvg) + k$.
\end{prop}

\begin{proof} 
Set $R = \kk[\vv^2]$. Let $S \subset \kvg$ be a separating set of size $r$ and let $J = (\delta(f): f \in S)R$. By Proposition \ref{radversion}, $\mathcal{S}_{G,\vv} = \spec(R/J)$.    

Now let $\mathfrak{m}$ be the maximal ideal of $R$ and let $\widehat{R}$ denote the $\mathfrak{m}$-adic completion of $R$. As $\vv^2$ is normal, we have that $\spec(\widehat{R})$ is irreducible, and hence connected in dimension $d$ for all $d \leq \dim(\spec(\widehat{R})) = 2n$. Applying the connectedness theorem with $d=2n-1$ shows that $\spec(\widehat{R}/J\widehat{R})$ is connected in dimension $2n-1-r$. 

Since $\widehat{R}/J\widehat{R}$ is also the completion of $R/J$ at the maximal ideal $\mathfrak{m}/J$, it follows from the proof of \cite[Lemma~4.3]{reimers} that $\spec((R/J)_{\mathfrak{m}/J})$ is connected in dimension $2n-1-r$, and from \cite[Proposition~4.4]{reimers} that $\spec((R/J) = \mathcal{S}_{G,\vv}$ is connected in dimension $2n-1-r$ too. Further since
\[\dim(\mathcal{S}_{G,\vv}) \geq \dim(\overline{\Gamma_{G,\vv}}) = 2n - \dim(\kvg)\] and all irreducible components intersect nontrivially we get that $\mathcal{S}_{G,V}$ is connected in codimension $1+r-\dim(\kvg)$. Since  $\mathcal{S}_{G,\vv}$ is not connected in codimension $k$ we must have that
\[k < 1+r-\dim(\kvg)\]
i.e. that $r \geq  \dim(\kvg)+k$ as required.

\end{proof}

\section{Invariants of $2 \times 2$ matrices}

We begin this section by fixing some notation and simplifying our problem as much as possible. As we will be considering only $2 \times 2$ matrices, we drop some subscripts, writing $\M = \M_2$, etc. Observe that the action of $\GL_2(\C)$ on $\M^n$ is not faithful - the kernel is the subgroup of scalar matrices. Therefore we have a faithful action of $\SL_2(\C)$ on $\M^n$ and $\C[\M^n]^{\GL_2(\C)} = \C[\M^n]^{\SL_2(\C)}$. Since it is enough to study this action, we write $G:= \SL_2(\C)$ from now on..

We fix the notation for some subgroups of $G$:  the torus
\[ T:= \left\{ \begin{pmatrix} t & 0 \\ 0 & t^{-1} \end{pmatrix}: t \in \C^*\right\},\]

the unipotent subgroup
\[ U:= \left\{ \begin{pmatrix} 1 & u \\ 0 & 1 \end{pmatrix}: u \in \C \right\},\]

and the Borel subgroup
\[ B:= \left\{ \begin{pmatrix} t & u \\ 0 & t^{-1} \end{pmatrix}: t \in \C^*, u \in \C \right\}.\]

Next observe that $\M$ is not an indecomposable representation of $G$; we have
\[\M \cong \vv \oplus \C I\]
where $\vv$ denotes the set of trace-zero matrices in $\M$, $I$ the  $2 \times 2$ identity matrix and $G$ acts trivially on $\C I$. It is now easy to see that 
\[\C[\M^n]^G = \C[\vv^n]^G \otimes \C[\Tr(X_i): i = 1, \ldots, n].\] Note also that $\vv$ is fixed by the commuting action of $\GL_n$ on $\M$ as defined in \eqref{commuting action}.

A generic element $\bA \in \vv^n$ will be written as $(A_1,A_2,\ldots,A_n)$ where
\[A_i = \begin{pmatrix} b_i & c_i \\ a_i & -b_i \end{pmatrix}.\]
A generic element $(\bA,\bA') \in \vv^n \times \vv^n$ will be written with $\bA$ as above and $\bA' = (A'_1,A'_2,\ldots,A'_n)$ where
\[A'_i = \begin{pmatrix} b'_i & c'_i \\ a'_i & -b'_i \end{pmatrix}.\]
Throughout we write $g \cdot \bA$ for the conjugation action of $g \in \SL_2(\C)$ on $\bA \in \M$. We use the notation $h \star \bA$ for the commuting action of $h \in \GL_n(\C)$.

From this point onwards, $X_i$ etc. represent generic trace-zero matrices of coordinate functions on $\vv$.
It is now easy to see that $\Tr(X^2_i) = -2 \det(X_i)$.
Consequently, Proposition \ref{lbp} implies:

\begin{prop}\label{gens} The following set $E_n$ of invariants minimally generates $\C[\vv^n]^G$ as an algebra:
\begin{itemize}

\item $t_{ij}:= \Tr(X_iX_j), 1 \leq i \leq j \leq n$.
\item $t_{ijk}:= \Tr(X_iX_jX_k), 1 \leq i<j<k \leq n$.
\end{itemize}
\end{prop}

We note for future reference that
\begin{equation}\label{tripletraceasdet} t_{ijk} = \begin{vmatrix} c_i & c_j & c_k \\ b_i & b_j & b_k \\ a_i & a_j & a_k \end{vmatrix}. 
\end{equation}

This implies that $t_{ijk} = \sgn(\sigma)t_{\sigma(i)\sigma(j)\sigma(k)}$ for any permutation $\sigma$ of $i,j,k$, and that $t_{iij} = 0$ for all $i,j$, etc.

The dimension of $\C[\vv^n]^G$ is $3n-3$ for $n \geq 2$. To see this note that there exists a trace-free $n$-matrix $\bA$ whose stabiliser is the finite group $\pm I$: take $$A_1 = \begin{pmatrix} 1 & 0 \\ 0 & -1 \end{pmatrix}, A_i = \begin{pmatrix} 0 & 1 \\ 0 & 0 \end{pmatrix}$$ for $i \geq 2.$  The result now follows from Corollary \ref{dimensionofinvariantring}. Note that this also implies that $\dim(\C[\M^n]^G) = 4n-3$ for $n \geq 2$ as stated in the introduction.

 The discussion above shows that our main results for $\C[\M^n]^G$ are equivalent to the following:
\begin{prop}\label{mainprop} Let $n \geq 2$ and suppose that $S \subseteq \C[\vv^n]^G$ is a separating set. Then $|S| \geq 4n-5$.
\end{prop}

\begin{prop}\label{mainprop2} Let $n \geq 3$. Then there exists a separating set $S' \subseteq \C[\vv^n]^G$ with cardinality $\frac12(n^2+7n-16)$.\end{prop}

We will need to consider certain subspaces of $\vv$: let $\ww$ denote the $B$-subspace of upper triangular matrices in $\vv$. Further, define
\[\cc:= \{(\bA,\bA') \in \ww^n \times \ww^n: b_i = b'_i \ \text{for all} \ i=1, \ldots, n\},\]
 \[\cc':= \{(\bA,\bA') \in \ww^n \times \ww^n: b_i =-b'_i \ \text{for all} \ i=1, \ldots, n\},\]
\[\cc_0:= \{(\bA,\bA') \in \ww^n \times \ww^n: b_i = b'_i= 0 \ \text{for all} \ i=1, \ldots, n\}.\]
The above are subspaces of $\vv^n \times \vv^n$, which are fixed under the diagonal action of $\GL_n(\C)$. Since the action commutes with the conjugation action of $G$, the same is true of the orbit subsets $(G \times G) \cdot \cc$, $(G \times G) \cdot \cc'$ and $(G \times G) \cdot \cc_0$.

Note that $U \cong \Ga(\C)$, the additive group of $\C$. The linear representation theory of this group is well-known: each indecomposable module is isomophic to $S^n(V)$, where $V$ is the restriction of the natural 2-dimensional $\C G$-module, and $S^n$ represents symmetric powers. One usually studies the so-called ``basic'' $\Ga$-actions: these are the $\Ga(\C)$-modules
\[V_i:= \langle v_0,v_1, \ldots, v_i \rangle: i \in \N\] on which $u \in \C$ acts via the formula
\[u \bullet v_i = \sum_{j=0}^i \frac{u^j}{j!} v_{i-j}\]
and it can be shown that $V_i \cong S^i(V)$. In our case a direct calculation shows that $\vv \cong V_2$ as a $U$-module.

The separating variety for arbitrary linear representations of $\Ga(\C)$ was considered by Dufresne and Kraft: in our case we have $\vv^n \cong V_2^{\oplus n}$ as $U$-modules, and from a careful reading of \cite[Theorem~7.5, Lemma~7.6]{DufresneKraft} we obtain:

\begin{prop}\label{uvar}\
\begin{itemize}
\item[(i)] We have $\mathcal{S}_{U,\vv^n} = \overline{\Gamma_{U,\vv^n}} \cup \cc$.
\item[(ii)] $\overline{\Gamma_{U, \vv^n}} = \Gamma_{U, \vv^n} \cup \cc'$.
\end{itemize}
In particular, $\cc_0 \subset \overline{\Gamma_{U, \vv^n}}$.
\end{prop} 

Note also that $T \cong \Gm(\C)$, the multiplicative group of $\C$. The representation theory of this group is very straightforward; the indecomposable modules are all 1-dimensional, isomorphic to $W_z$ for some $z \in \Z$ where $W_z = \langle v \rangle$ and the action of $t \in T$ is given by
\[t \bullet v = t^z v.\]
Now it's easy to see that $\vv \cong V_0 \oplus V_0 \oplus V_2 \oplus V_{-2}$. Moreover, $\ww$ is a direct summand of $\vv$ isomorphic to $V_0 \oplus V_0 \oplus V_2$.

The separating variety for arbitary linear actions of algebraic tori $\Gm^k$ was considered by Dufresne and Jeffries. In our case the action is particularly simple, and from careful reading of \cite[Lemma~3.5]{DufresneJeffriesTori} we obtain:

\begin{prop}\label{tvar}\
\begin{itemize}
\item[(i)] We have $\mathcal{S}_{T,\ww^n} = \cc.$

\item[(ii)] $(\bA,\bA') \in \overline{\Gamma_{T,\ww^n}}$ if and only if $(\bA,\bA') \in \mathcal{C}$ and in addition $c_ic'_j = c_jc'_i$ for all $1 \leq i,j \leq n$.
\end{itemize}
\end{prop}

Our goal in this section is to describe the separating variety $\mathcal{S}_{G,\vv^n}$. As observed in the previous section, one irreducible component is given by the the Zariski closure of the graph 
$$\Gamma_{G,\vv^n} = \{(\bA,g \cdot\bA): \bA \in \vv^n, g \in G\}$$
but there may be more. The description of $\mathcal{S}_{G,\vv^n}$ as the set of pairs of $n$-matrices whose orbit closures intersect makes it clear that the separating variety is fixed by the action of $G \times G$. It is also clear that $\Gamma_{G,\vv^n}$ is fixed by the action of $G \times G$. The commuting action of $\GL_n(\C)$ on $\vv^n$ induces a diagonal action on $\vv^n \times \vv^n$ which commutes with the action of $G \times G$. The following simple observations are crucial:

\begin{Lemma}\label{reducegamma} Suppose $(\bA,\bA') \in \overline{\Gamma_{G,\vv^n}}$. Then $(h \star \bA, h \star \bA') \in \overline{\Gamma_{\vv^n}}$ for all $h \in \GL_n(\C)$.
\end{Lemma}

\begin{proof} Let $(\bA,\bA') \in \overline{\Gamma_{G,\vv^n}}$. Then there exist morphisms $\bA: \C^* \rightarrow \vv^n$, $g: \C^* \rightarrow G$ such that
\[\bA = \lim_{t \rightarrow 0} \bA(t), \bA' = \lim_{t \rightarrow 0} g(t) \cdot \bA(t).\]
Let $h \in \GL_n(\C)$, then
\[h \star \bA = \lim_{t \rightarrow 0} (h \star \bA(t)), \bA' = \lim_{t \rightarrow 0}h \star( g(t) \cdot \bA(t)) = \lim_{t \rightarrow 0}  g(t) \cdot(h \star \bA(t))).\]
\end{proof}

\begin{Lemma}\label{reducesepvar} Let $(\bA,\bA') \in \mathcal{S}_{G,\vv^n}$ and let $h \in \GL_n(\C)$. Then $(h \star \bA, h \star \bA') \in \mathcal{S}_{G,\vv^n}$.  
\end{Lemma}

\begin{proof} The action of $\GL_n(\C)$  on $\vv^n$ induces an action on $\C[\vv^n]$: we define
\[(h \star f)(v) = f(h^{-1} \star \bA)\] for $h \in \GL_n(\C), f \in \C[\vv^n]$ and $\bA \in \vv^n$. If $f \in \C[\vv^n]^G$, then so is $h \star f$ for all $h \in \GL_n(\C)$. Now suppose $(\bA,\bA') \in \mathcal{S}_{G,\vv^n}$. Let $h \in \GL_n(\C)$ and $f \in \C[\vv^n]^G$. Then we have
\[f(h \star \bA) = (h^{-1} \star f)(\bA) =(h^{-1} \star f)(\bA') = f(h \star \bA')\] which shows that $(h \star \bA, h \star \bA') \in \mathcal{S}_{G,\vv^n}$ as required.
\end{proof}

These observations allow us to apply a sort of simultaneous column-reduction to elements of $\cc$ and $\cc'$. In more detail, to every pair $(\bA,\bA') \in \cc$ or $\cc'$ we assign a $3 \times n$ matrix
\begin{equation}\label{manddelta} m_{\bA,\bA'}:= \begin{pmatrix} b_1 & b_2 & \cdots &b_n\\ c_1 & c_2 & \cdots & c_n\\ c'_1 & c'_2 & \cdots & c'_n   \end{pmatrix}.\end{equation} 
Denote the minor of $m_{\bA,\bA'}$ obtained by taking only the $i$th, $j$th and $k$th columns by $\Delta_{ijk}(\bA,\bA')$. Finally we define
\[\widehat{\cc}:= \{(\bA,\bA') \in \cc: \rk(m_{\bA,\bA'}) \leq 2\},\]
and
\[\widehat{\cc'}:= \{(\bA,\bA') \in \cc': \rk(m_{\bA,\bA'}) \leq 2\}.\]

Then we have:
\begin{Lemma}\label{reducecols} The following are equivalent:
\begin{enumerate}
\item[(i)] $\rk(m_{\bA.\bA'}) \leq 2$;
\item[(ii)] $\Delta_{ijk}(\bA,\bA') = 0$ for all $1 \leq i<j<k \leq n$;
\item[(iii)] There exists $h \in \GL_n$ such that $(h \star  A_i) = (h \star A'_i) = 0$ for all $i \geq 3$. 

\iffalse %not used
Moreover, we can assume that either
\begin{equation}\label{firstform}  h \star A_1 = \begin{pmatrix} 0 & c_1 \\ 0 & 0 \end{pmatrix}, h \star A_2 = \begin{pmatrix} 1 & 0 \\ 0 & -1 \end{pmatrix}, h \star A'_1 = \begin{pmatrix} 0 & c'_1 \\ 0 & 0 \end{pmatrix},h \star A'_2 = \begin{pmatrix} 1 & 0 \\ 0 & -1 \end{pmatrix}\end{equation}
or \begin{equation}\label{secondform}  h \star A_1 = \begin{pmatrix} b_1 & 1 \\ 0 & -b_1 \end{pmatrix}, h \star A_2 = \begin{pmatrix} b_2 & 0 \\ 0 & -b_2 \end{pmatrix}, h \star A'_1 = \begin{pmatrix} b_1 & 0 \\ 0 & -b_1 \end{pmatrix},h \star A'_2 = \begin{pmatrix} b_2 & 1 \\ 0 & -b_2 \end{pmatrix},\end{equation} in the case that $(\bA,\bA') \in \cc$, with the diagonals swapped over in $h \star A'_i$ in case $(\bA,\bA') \in \cc'$, or $h \star A_2 = h\star A'_2 = 0$ (in the rank $\leq 1$ case).
\fi
\end{enumerate}
\end{Lemma}

\begin{proof} The equivalence of (i) and (ii) is well-known. Let $E_{ij}(\lambda)$ denote the elementary matrix with $\lambda$ in position $i,j$, 1's on the diagonal and zeroes elsewhere. The action of $E_{ij}(\lambda)$ on $(\bA,\bA')$ is to replace $A_i$ and $A'_i$ by $A_i+\lambda A_j$ and $A'_i+\lambda A_j' $ respectively. The effect on $m_{\bA,\bA'}$ is to replace the $i$th column with the $i$th column plus $\lambda$ times the $j$th column. The sequence of column operations which place $m_{\bA,\bA'}$ in column echelon form produce  $(h \star \bA, h \star\bA')$ of the form claimed.
\end{proof}

Now we consider the $G$-orbit structure of $\vv$. The following result is key to obtaining a decomposition of the separating variety $\mathcal{S}_{G,\vv^n}$:

\begin{Lemma}\label{orbits}
Suppose $\bA \in \vv^n \setminus G  \cdot \ww^n$. Then $G \cdot \bA$ is closed
\end{Lemma}

\begin{proof} We use the Hilbert-Mumford criterion \cite{Mumford}. Let $\lambda: \C^* \rightarrow G$ be a one-parameter subgroup. The Hilbert-Mumford weight of $\bA$ with respect to $\lambda$ is the unique smallest integer $\mu(\bA,\lambda)$ such that $$\lim_{t \rightarrow 0} t^{\mu(\bA,\lambda)} (\lambda(t) \cdot \bA)$$ exists. $\bA$ is stable (i.e. has closed orbit and finite stabiliser) if and only if $\mu(\bA,\lambda)>0$ for all one-paramter subgroups $\lambda$ of $G$. Now each one-parameter subgroup of $G$ is of the form
\[\lambda_g(t) = g \cdot \begin{pmatrix} t& 0 \\ 0 & t^{-1} \end{pmatrix}\]
for some $g \in G$. Letting $e \in G$ denote the identity, it follows that $\bA$ is stable if and only if $\mu(g \cdot \bA,\lambda_e)>0$ for all $g \in G$. 

Let $g \in G$ and suppose $g \cdot \bA \not \in \ww^n$. Writing $g \cdot \bA = \bA'$, we see that $a'_i \neq 0$ for some $i=1, \ldots, n$. Therefore $\mu(g \cdot \bA,\lambda_e) = 2> 0$. Since $g$ was arbitrary we conclude that $\bA$ is stable, and in particular $G \cdot \bA$ is closed.
\end{proof}

\begin{rem}
One can complete the analysis of orbit closures in $\vv^n$: it turns out that $G \cdot \bA$ is not closed if and only if $\bA \in G \cdot \ww^n  \setminus \{\mathbf{0}\}$ where $\mathbf{0}$ represents the $n$-matrix consisting entirely of zero matrices. See \cite[Proposition~8.9]{DrezetLuna} for a proof in a more general case.
\end{rem}

An $n$-matrix with all traces equal to zero belongs to $G \cdot \ww^n$ if and only if there exists $g \in G$ such that $g \cdot A_i$ is upper triangular for all $i$. Such matrices may be called {\it simultaneously (upper)-triangularisable}. Since an matrix is upper triangularisable if and only if it is lower triangularisable via the action of $$\begin{pmatrix} 0 & 1 \\ -1 & 0 \end{pmatrix}$$ we usually speak simply of simultaneously triangularisable matrices.

Simultaneous triangularisation of $2 \times 2$ matrices was studied by Florentino \cite{florentino2x2}. He showed that arbitrary  (i.e. not trace-free) $n$-matrices are upper triangularisable if and only if
\[\Tr(A_iA_jA_k) = \Tr(A_kA_jA_i)\] for all $i<j<k$ and 
\[\det(A_iA_j-A_jA_i) = 0\] for all $i<j$.
But for tracefree matrices, Equation \eqref{tripletraceasdet} implies that $$\Tr(A_iA_jA_k) = -\Tr(A_kA_jA_i).$$ Further, for tracefree matrices we can show \[ \det(A_iA_j-A_jA_i) =   4\det (A_i) \det(A_j) - \Tr(A_iA_j)^2.\] It follows that $G \cdot \ww^n$ is the subvariety of $\vv^n$ cut out by the set of polynomials
\[\{t_{ijk}: 1 \leq i<j<k \leq n\} \cup \{t_{ii}t_{jj} - t_{ij}^2: 1 \leq i<j \leq n\}.\]

In particular, this shows that $G \cdot \ww^n$ is closed. It is also irreducible, being the orbit of a connected algebraic group on a vector space. As a first step towards a decomposition of $\mathcal{S}_{G,\vv^n}$ we have the following. 

\begin{Lemma}\label{decompsgv}$\mathcal{S}_{G,\vv} = \overline{\Gamma_{G,\vv}} \cup ((G \cdot \ww^n \times G \cdot \ww^n) \cap \mathcal{S}_{G,\vv^n})$. Moreover, $\overline{\Gamma_{G,\vv^n}}$ and $(G \cdot \ww^n \times G \cdot \ww^n) \cap \mathcal{S}_{G,\vv^n}$ are closed and $\overline{\Gamma_{G,\vv^n}}$ is  irreducible.
\end{Lemma}

\begin{proof}
Let $(\bA,\bA') \in \mathcal{S}_{G,\vv^n}$. Then $\overline{G \cdot \bA} \cap \overline{G \cdot \bA'} \neq \emptyset$. Unless $\bA \in G \cdot \ww^n$ and $\bA' \in G \cdot \ww^n$, then by Lemma \ref{orbits}, the orbit of either $\bA$ or $\bA'$ is closed and $(\bA,\bA') \in \overline{\Gamma_{G,\vv^n}}$ (if both orbits are closed, $(\bA,\bA') \in \Gamma_{G,\vv^n}).$ This shows that $\mathcal{S}_{G,\vv^n} \subseteq \overline{\Gamma_{G,\vv^n}} \cup  ((G \cdot \ww^n \times G \cdot \ww^n) \cap \mathcal{S}_{G,\vv^n})$. Clearly $\overline{\Gamma_{G,\vv^n}}$ is closed and irreducible. $(G \cdot \ww^n \times G \cdot \ww^n) \cap \mathcal{S}_{G,\vv^n}$ is closed as it is the intersection of two closed sets.
\end{proof}

The following Lemma describes $(G \cdot \ww^n \times G \cdot \ww^n) \cap \mathcal{S}_{G,\vv^n}$ as a union of sets. Note that at this stage we do not know whether $(G \cdot \ww^n \times G \cdot \ww^n) \cap \mathcal{S}_{G,\vv^n}$ is irreducible, as the sets on the right hand side may not themselves be closed.  

\begin{Lemma}\label{decompgwgwsgv} We have $$(G \cdot \ww^n \times G \cdot \ww^n) \cap \mathcal{S}_{G,\vv^n} = (G \times G) \cdot \cc \cup (G \times G) \cdot \cc'.$$
\end{Lemma}

\begin{proof} 
Suppose $(\bA, \bA') \in (\ww^n \times \ww^n) \cap \mathcal{S}_{G,\vv^n}$. Then since $\det(A_i) = \det(A'_i)$ for all $i$ we get $b_i = \pm b'_i$. Further, since for each $i \neq j$
\begin{equation}\label{2trace} 2b_ib_j  = \Tr(A_iA_j) = \Tr(A'_iA'_j) = 2b'_ib'_j\end{equation} we have that $b_i=b'_i$ for each $i$ or $b_i = -b_i'$ for each $i$. This shows that  
$(\bA,\bA') \in \cc \cup \cc'$. Since $(\bA,\bA') \in \mathcal{S}_{G,\vv^n}$ if and only $(g \cdot \bA, g' \cdot \bA') \in \mathcal{S}_{G,\vv^n}$ for all $g,g' \in G$ we get that $$(G \cdot \ww^n \times G \cdot \ww^n) \cap \mathcal{S}_{G,V} = (G \times G) \cdot \cc \cup (G \times G) \cdot \cc'$$
 as required.
\end{proof}

\begin{Lemma}\label{losecminus} For all $n \geq 2$ we have $(G \times G) \cdot \cc' \subseteq \overline{\Gamma_{G,\vv^n}}$.
\end{Lemma}

\begin{proof}  As $\overline{\Gamma_{G, \vv^n}}$ is stablised by the action of $G \times G$, it is enough to show that $\cc' \subseteq \overline{\Gamma_{G,\vv^n}}$, and this follows from Proposition \ref{uvar} (ii): we have
\[\cc' \subseteq \overline{\Gamma_{U,\vv^n}} \subseteq \overline{\Gamma_{G,\vv^n}}\] as required.
\end{proof}

Combining Lemmas \ref{decompsgv}, \ref{decompgwgwsgv} and \ref{losecminus} we obtain a decomposition
\begin{equation}\label{decompsgvagain}
\mathcal{S}_{G,\vv^n} = \overline{\Gamma_{G,\vv^n}} \cup \overline{(G \times G) \cdot \cc}.
\end{equation}
Note that the stabiliser of the action of $G \times G$ on $\cc$ is $B \times B$, so the dimension of $\overline{(G \times G) \cdot \cc}$ is 
\[\dim(\cc) + 2 \dim(G) - 2 \dim(B) = 3n+2. \]

Since $G$ acts faithfully on $\vv^n$, the dimension of $\overline{\Gamma_{G,\vv^n}}$ is
$$\dim(\vv^n) +\dim(G) = 3n+3.$$

So there are two possiblities: either$\overline{(G \times G) \cdot \cc} \subseteq \overline{\Gamma_{G, \vv^n}}$ and hence $\mathcal{S}_{G,\vv^n}$ has a single irreducible component, or else $\mathcal{S}_{G,\vv^n}$ has two irreducible components of different dimensions. 

The following result is key in our proof of Proposition \ref{mainprop}.

%%%%%%%%%%%%%%%%%%%%%%%%%%%%%%%%%%%%%%%%%%%%%%%%%%%%%%%%

\begin{prop}\label{intersect}\
For $n \geq 2$ we have $(G \times G) \cdot \cc \cap \overline{\Gamma_{G,\vv^n}} = (G \times G) \cdot \widehat{\mathcal{C}}$.
\end{prop}
 
Note that for $n=2$ this implies that $(G \times G) \cdot \cc \subseteq \overline{\Gamma_{G,\vv^n}}$ and hence $\overline{(G \times G) \cdot \cc} \subseteq \overline{\Gamma_{G,\vv^n}}$, since $\widehat{\cc} =\cc$. Therefore it follows that $\mathcal{S}_{G,\vv^2}$ is irreducible.

\begin{proof} We begin by showing that $(G \times G)\cdot \widehat{\cc} \subseteq (G \times G) \cdot \cc \cap \overline{\Gamma_{G,\vv^n}}$. It is clear that $(G \times G) \cdot \widehat{\cc} \subseteq (G \times G) \cdot \cc$ and since $\overline{\Gamma_{G,\vv^n}}$ is fixed by the action of $G \times G$ it is enough to show that $\widehat{\cc} \subseteq \overline{\Gamma_{G,\vv^n}}$. 

Let $(\bA,\bA') \in \widehat{\cc}$. By Lemma \ref{reducegamma} it is enough to show that there exists $h \in \GL_n(\C)$ with $(h \star \bA, h\star \bA') \in \overline{\Gamma_{G,\vv^n}}$. Therefore by Lemma \ref{reducecols} we may assume $A_i = A'_i = 0$ for $i \geq 3.$ 

Suppose that there exists $i$ such that $b_i=b_i' \neq 0$, without loss of generality $i=1$. Set
\begin{equation}\label{chooseg} g:= \begin{pmatrix} 1 & \frac{-c_1}{2b_1} \\0 & 1 \end{pmatrix}, g':= \begin{pmatrix} 1 & \frac{-c'_1}{2b_1}\\ 0 & 1 \end{pmatrix}.\end{equation}
Then $$g \cdot A_1 = g' \cdot A_1' = \begin{pmatrix} b_1&0 \\0 & -b_1 \end{pmatrix},$$ and $(g \cdot A_2, g' \cdot A'_2)$ are upper triangular with equal diagonal entries. By Proposition \ref{tvar}, $$(g \cdot \bA, g' \cdot \bA') \in \overline{\Gamma_{T,\ww^n}} \subseteq \overline{\Gamma_{G,\vv^n}}.$$ Therefore $(\bA,\bA') \in \overline{\Gamma_{G,\vv^n}}$ as required.

Hence we may assume $b_i=b_i'=0$ for $i=1,2$. But in that case 
$$(\bA,\bA') \in \cc_0 \subseteq \overline{\Gamma_{U,\vv^n}} \subseteq  \overline{\Gamma_{G,\vv^n}}$$ where the first inclusion comes from Proposition \ref{uvar}.

To prove the converse it is enough to show that $\cc \cap \overline{\Gamma_{G,\vv^n}} \subseteq \widehat{\mathcal{C}}$. So, let $(\bA,\bA') \in \cc \cap \overline{\Gamma_{G,\vv^n}}$. Thus, there exist morphisms $g: \C^* \rightarrow G$, and $\bA: \C^* \rightarrow \vv^n$ such that
\[\lim_{t \rightarrow 0} \bA(t) = \bA, \lim_{t \rightarrow 0}(g(t) \cdot \bA(t)) = \bA',\] where we abuse notation by using $\bA$ for a function and its limit. Also let $\bA'(t):= (g(t) \cdot \bA(t))$ for all $t \in \C^*$. Note that although $a_i = a'_i = 0$ for all $i$ and $b_i=b'_i$ we do not have $a_i(t) = a'_i(t) = 0$ or $b_i(t) = b'_i(t)$ for all $t$ in general. 
Write \[g(t) = \begin{pmatrix} w(t) & x(t) \\ y(t) & z(t) \end{pmatrix}.\]
Let $1 \leq i<j<k \leq n$. Write $\ba(t),\bb(t),\bc(t)$ and $\bc(t)'$ for the row vectors $(a_i(t),a_j(t),a_k(t))$,$(b_i(t),b_j(t),b_k(t))$, $(c_i(t),c_j(t),c_k(t))$ and $(c'_i(t),c'_j(t),c'_k(t))$ respectively with analogous notation for their limits, and consider $$\Delta_{ijk}(\bA,\bA') = \begin{vmatrix} \bb\\ \bc \\ \bc'\end{vmatrix}.$$ On the one hand this is
\begin{eqnarray*} & \lim_{t \rightarrow 0} \begin{vmatrix} \bb(t)\\ \bc(t) \\ \bc'(t)\end{vmatrix}\\
&= \lim_{t \rightarrow 0} \begin{vmatrix} \bb(t) \\ \bc(t) \\ 2x(t)z(t)\bb(t) - x(t)^2\ba(t)+z(t)^2\bc(t) \end{vmatrix}\\
&= -\lim_{t \rightarrow 0} x(t)^2 \begin{vmatrix} \bb(t) \\ \bc(t) \\ \ba(t) \end{vmatrix}.
\end{eqnarray*}
using row operations and the facts that $\bb(t) \rightarrow \bb$, $\bc(t) \rightarrow \bc$, $\bc'(t) \rightarrow \bc'$ as $t \rightarrow 0$. But we also have $\bb'(t) \rightarrow \bb$, so that 
\begin{eqnarray*} \Delta_{ijk}(\bA,\bA')&= \lim_{t \rightarrow 0} \begin{vmatrix} \bb'(t)\\ \bc(t) \\ \bc'(t)\end{vmatrix}\\
&= \lim_{t \rightarrow 0} \begin{vmatrix} 2x(t)y(t)\bb(t) - w(t)x(t)\ba(t) + y(t)z(t)\bc(t) \\ \bc(t) \\ 2x(t)z(t)\bb(t) - x(t)^2\ba(t)+z(t)^2\bc(t) \end{vmatrix}\\
&= \lim_{t \rightarrow 0}  \begin{vmatrix} 2x(t)y(t)\bb(t) - w(t)x(t)\ba(t) \\ \bc(t) \\ 2z(t)\bb(t)x(t) - x(t)^2\ba(t) \end{vmatrix}.
\end{eqnarray*}

Noting that $w(t)z(t)-x(t)y(t)=1$ for all $t \in \C^*$, this is equal to 
\begin{eqnarray*} &\lim_{t \rightarrow 0} \begin{vmatrix} 2(w(t)z(t)-1)\bb(t) - w(t)x(t)\ba(t) \\ \bc(t) \\ 2z(t)\bb(t)x(t) - x(t)^2\ba(t) \end{vmatrix}\\
&= \lim_{t \rightarrow 0}x(t) \begin{vmatrix} 2(w(t)z(t)-1)\bb(t) - w(t)x(t)\ba(t) \\ \bc(t) \\ 2z(t)\bb(t) - x(t)\ba(t) \end{vmatrix}\\
&=  \lim_{t \rightarrow 0}x(t) \begin{vmatrix} w(t)(2z(t)\bb(t)-x(t)\ba(t)) - 2\bb(t) \\ \bc(t) \\ 2z(t)\bb(t) - x(t)\ba(t) \end{vmatrix}\\
&=  \lim_{t \rightarrow 0}x(t) \begin{vmatrix} - 2\bb(t) \\ \bc(t) \\ 2z(t)\bb(t) - x(t)\ba(t) \end{vmatrix}\\
&=  \lim_{t \rightarrow 0} 2x(t)^2 \begin{vmatrix} \bb(t) \\ \bc(t) \\ \ba(t) \end{vmatrix}\\
\end{eqnarray*}

This shows that $\Delta_{ijk}(\bA,\bA') = 0$ as required, and therefore $(\bA,\bA') \in \widehat{\cc}$.
\end{proof}

\begin{eg} Let $n \geq 3$ and define a pair of $n$-matrices:
\[A_1 = \begin{pmatrix} 1 & 0\\ 0 & -1 \end{pmatrix}, A_2 = \begin{pmatrix}1 & 1\\0& -1 \end{pmatrix}, A_3 = \begin{pmatrix} 1&1\\ 0 & -1 \end{pmatrix}\] and
\[A'_1 = \begin{pmatrix} 1 & 0\\ 0 & -1 \end{pmatrix}, A'_2 = \begin{pmatrix}1 & 0\\0& -1 \end{pmatrix}, A'_3 = \begin{pmatrix} 1&1\\ 0 & -1 \end{pmatrix}\]
with $A_i=A'_i = 0$ for all $i>3$. Then $\rk(m_{\bA,\bA'})=3$, and so proposition \ref{intersect} above shows that $(\bA,\bA') \in \overline{(G \times G) \cdot \cc } \setminus \overline{\Gamma_{G,\vv^n}}$. Thus, for $n \geq 3$, $\mathcal{S}_{G,\vv^n}$ has two irreducible components.
\end{eg}

We obtain the following immediate Corollary, from which Corollary \ref{maincor} also follows:

\begin{cor} Let $n \geq 3$ and suppose $S \subseteq \C[\vv^n]^G$ is a separating set. Then $|S| \geq 3n-2$.
\end{cor}

\begin{proof}
For $n\geq 3$ we have shown that $\mathcal{S}_{G,\vv^n}$ has two irreducible components. The component $\overline{\Gamma_{G,\vv^n}}$ has dimension $3n+3$,  i.e. codimension $3n-3$ in $\vv^n \times \vv^n$, while $(G \cdot \ww^n \times G \cdot \ww^n) \cap \mathcal{S}_{G,\vv^n} $ has dimension $3n+2$, i.e. codimension $3n-2$ in $\vv^n \times \vv^n$. The result now follows from Proposition \ref{krullbound}.
\end{proof}

%%%%%%%%%%%%%%%%%%%%%%%%%%5

To obtain the stronger bound in Proposition \ref{mainprop}, we need to compute the dimension of $\overline{(G \times G) \cdot \cc} \cap \overline{\Gamma_{G,\vv^n}}$. We are not able to find a complete description of $\overline{(G \times G) \cdot \cc}$; but for our purposes the following is enough.

\begin{Lemma}\label{closureggc} \[\overline{(G \times G) \cdot \cc} \subseteq (G \times G) \cdot \cc \cup (G \times G)\cdot\widehat{\cc'}.\]
\end{Lemma}

\begin{proof} Suppose $(\bA,\bA') \in \overline{(G \times G) \cdot \cc} \setminus (G \times G) \cdot \cc$. As $(G \times G) \cdot \cc \subseteq (G \cdot \ww^n \times G \cdot \ww^n) \cap \mathcal{S}_{G,\vv^n}$ and the latter is closed, we have
\[(\bA,\bA') \in \overline{(G \times G) \cdot \cc} \subseteq  (G \cdot \ww^n \times G \cdot \ww^n) \cap \mathcal{S}_{G,\vv^n} = (G \times G) \cdot \cc \cup (G \times G) \cdot \cc',\] with the final equality coming from Lemma \ref{decompgwgwsgv}. Therefore $(\bA,\bA') \in (G \times G) \cdot \cc'$. Since $(G \times G) \cdot \cc$, $\overline{(G \times G) \cdot \cc}$ and $(G \times G) \cdot \cc'$ are all fixed by the action of $(G \times G)$, it is enough to show that
\[\cc' \cap \overline{(G \times G) \cdot \cc} \subseteq \widehat{\cc'}.\]

So, suppose $(\bA,\bA') \in \overline{(G \times G) \cdot \cc}  \cap \cc'$; we aim to show that $\rk(m_{\bA,\bA'}) \leq 2$.

 Let $1 \leq i<j<k \leq n$, and write $\bb$, $\bc$ and $\bc'$ for the row vectors $(b_i,b_j,b_k)$, $(c_i,c_j,c_k)$ and $(c'_i,c'_j,c'_k)$ respectively. Since $(\bA,\bA') \in \overline{(G \times G) \cdot \cc}$, there exist morphisms $g, g': \C^* \rightarrow G$, and $\bZ, \bZ': \C^* \rightarrow \cc$ such that
\[\lim_{t \rightarrow 0} g(t) \cdot \bZ(t) = \bA, \lim_{t \rightarrow 0} g'(t) \cdot \bZ'(t) = \bA'.\]
Write \[g(t) = \begin{pmatrix} w(t) & x(t) \\ y(t) & z(t) \end{pmatrix}, g'(t) = \begin{pmatrix} w'(t) & x'(t) \\ y'(t) & z'(t) \end{pmatrix},\]
and
\[Z_l(t) = \begin{pmatrix} \beta_l(t) & \gamma_l(t) \\ 0 & -\beta_l(t) \end{pmatrix}, Z'_l(t) = \begin{pmatrix} \beta_l(t) & \gamma_l'(t) \\ 0 & -\beta_l(t) \end{pmatrix}\] for $l=1, \ldots, n$, where $\bZ = (Z_1,Z_2, \ldots, Z_n)$ and $\bZ'$ is similar.

Now writing $\bbeta$, $\bgamma$ and $\bgamma'$ for the row vectors $(\beta_i,\beta_j,\beta_k)$, $(\gamma_i,\gamma_j,\gamma_k)$ and $(\gamma'_i,\gamma'_j,\gamma'_k)$ respectively and evaluating $g(t) \cdot \bZ(t)$ shows that 
\begin{equation}\label{bb} \bb = \lim_{t \rightarrow 0} (w(t)z(t)\bbeta(t) - w(t)y(t)\bgamma(t) + x(t)y(t)\bbeta(t)) \end{equation}
and
\begin{equation}\label{bc} \bc = \lim_{t \rightarrow 0} (-2x(t)w(t)\bbeta(t)+w(t)^2\bgamma(t)). \end{equation}
Meanwhile, evaluating $g'(t) \cdot \bZ'(t)$ shows that 
\begin{equation}\label{bb'} \bb = \lim_{t \rightarrow 0}(  -w'(t)z'(t)\bbeta(t) + w'(t)y'(t)\bgamma'(t) - x(t)y(t)\bbeta(t)) \end{equation}
and
\begin{equation}\label{bc'} \bc' = \lim_{t \rightarrow 0} (-2x'(t)w'(t)\bbeta(t)+w'(t)^2\bgamma'(t)). \end{equation}

Now using \eqref{bb}, \eqref{bc} and \eqref{bc'} we have 
\begin{eqnarray*} \Delta_{ijk}(\bA,\bA') &=& \lim_{t \rightarrow 0}\begin{vmatrix} w(t) z(t) \bbeta(t)-w(t)y(t)\bgamma(t)+x(t)y(t)\bbeta(t) \\ -2w(t)x(t)\bbeta(t) + w(t)^2 \bgamma(t)\\ -2w'(t)x'(t)\bbeta(t)+w'(t)^2\bgamma(t) \end{vmatrix}\\
&=& \lim_{t \rightarrow 0} w(t)w'(t)\begin{vmatrix} \bbeta(t)+2x(t)y(t)\bbeta(t)-w(t)y(t)\bgamma(t)\\  -2x(t)\bbeta(t) + w(t) \bgamma(t)\\ -2x'(t)\bbeta(t)+w'(t)\bgamma(t) \end{vmatrix}
\end{eqnarray*}
where we use the fact that $w(t)z(t)-x(t)y(t)=1$ for all $t \in \C^*$ in the first row. Now using row operations to simplify the determinant we get
\begin{eqnarray*} \Delta_{ijk}(\bA,\bA') &=&  \lim_{t \rightarrow 0} w(t)w'(t)\begin{vmatrix} \bbeta(t)\\  -2x(t)\bbeta(t) + w(t) \bgamma(t)\\ -2x'(t)\bbeta(t)+w'(t)\bgamma'(t) \end{vmatrix}\\
&=& \lim_{t \rightarrow 0} w(t)^2w'(t)^2\begin{vmatrix} \bbeta(t)\\  \bgamma(t)\\ \bgamma'(t) \end{vmatrix}.
\end{eqnarray*}

On the other hand, using  \eqref{bc}, \eqref{bb'} and \eqref{bc'} gives us 
\begin{eqnarray*} \Delta_{ijk}(\bA,\bA') &=& \lim_{t \rightarrow 0}\begin{vmatrix} -w'(t) z'(t) \bbeta(t)+w'(t)y'(t)\bgamma'(t)-x'(t)y'(t)\bbeta(t) \\ -2w(t)x(t)\bbeta(t) + w(t)^2 \bgamma(t)\\ -2w'(t)x'(t)\bbeta(t)+w'(t)^2\bgamma(t) \end{vmatrix}\\
&=& \lim_{t \rightarrow 0} w(t)w'(t)\begin{vmatrix} -\bbeta(t)-2x'(t)y'(t)\bbeta(t)+w'(t)y'(t)\bgamma'(t)\\  -2x(t)\bbeta(t) + w(t) \bgamma(t)\\ -2x'(t)\bbeta(t)+w'(t)\bgamma(t) \end{vmatrix}
\end{eqnarray*}
where we use the fact that $w'(t)z'(t)-x'(t)y'(t)=1$ for all $t \in \C^*$ in the first row. Now using row operations to simplify the determinant we get
\begin{eqnarray*} \Delta_{ijk}(\bA,\bA') &=&  \lim_{t \rightarrow 0} w(t)w'(t)\begin{vmatrix} -\bbeta(t)\\  -2x(t)\bbeta(t) + w(t) \bgamma(t)\\ -2x'(t)\bbeta(t)+w'(t)\bgamma'(t) \end{vmatrix}\\
&=& \lim_{t \rightarrow 0} w(t)^2w'(t)^2\begin{vmatrix} -\bbeta(t)\\  \bgamma(t)\\ \bgamma'(t) \end{vmatrix}.
\end{eqnarray*}
This shows that $\Delta_{ijk}(\bA,\bA') = 0$. Since this applies to all $1 \leq i<j<k \leq n$ we get $\rk(m_{\bA,\bA'}) \leq 2$ as required, i.e. $(\bA,\bA') \in \widehat{\cc'}$.
\end{proof}

\begin{cor}\label{dimintersect} We have $$\overline{(G \times G)\widehat{\cc}} \subseteq \overline{(G \times G) \cc} \cap \overline{\Gamma_{G,\vv^n}} \subseteq \overline{(G \times G)\widehat{\cc}} \cup \overline{(G \times G)\widehat{\cc'}}.$$
\end{cor}

\begin{proof} The first inclusion follows from Proposition \ref{intersect} and the fact that the middle term is closed. For the second we have 
\[ \overline{(G \times G) \cc} \cap \overline{\Gamma_{G,\vv^n}} \subseteq ((G \times G) \cdot \cc \cup (G \times G)\widehat{\cc'}) \cap \overline{\Gamma_{G,\vv^n}} \] by Lemma \ref{closureggc},
\[= (G \times G) \widehat{\cc} \cup (\overline{\Gamma_{G,\vv^n}} \cap (G \times G)\widehat{\cc'})  \]
by Proposition \ref{intersect},
\[= (G \times G) \widehat{\cc} \cup (G \times G)\widehat{\cc'} \] since $$(G \times G)\widehat{\cc'} \subseteq (G \times G) \cdot \cc' \subseteq \overline{\Gamma_{G,\vv^n}}$$ by Lemma \ref{losecminus}.
\end{proof}

Notice that $\widehat{\cc}$ and $\widehat{\cc'}$ are both fixed by $B \times B$, since $\Delta_{ijk}$ is $B \times B$ semi-invariant for all $i,j,k$ on both $\cc$ and $\cc'$, and both $\cc$ and $\cc'$ are $B \times B$-modules. Therefore we have
\[\dim(\overline{(G \times G) \widehat{\cc}}) = \dim(\widehat{\cc}) + 2 \dim(G) - 2 \dim(B) =  \dim(\widehat{\cc}) + 2.\] 
Now it's easy to see that $\widehat{\cc}$ is isomorphic to the variety of $3 \times n$ matrices with rank at most 2, which is well known to have dimension $2n+2$. Therefore we have $\dim(\overline{(G \times G) \widehat{\cc}}) = 2n+4.$ Exactly the same argument also shows $\dim(\overline{(G \times G) \widehat{\cc'}}) = 2n+4.$ Now Corollary \ref{dimintersect} above shows that $$\dim(\overline{(G \times G) \cdot \cc} \cap \overline{\Gamma_{G,\vv^n}}) = 2n+4.$$

Consequently, since $\dim(\mathcal{S}_{G,\vv^n}) = 6n - (3n-3) = 3n+3$, we get that $\mathcal{S}_{G,\vv^n}$ is connected in codimension $n-1$ but not in codimension $n-2$. We also note that $\mathcal{S}_{G,\vv^n}$ contains just two components and these have non-trivial intersection. By Proposition \ref{grothbound}, for any separating set $S \subseteq \C[\vv^n]^G$ we have
\[|S| \geq \dim(\C[\vv^n])^G + n - 2 = 3n-3+n - 2 = 4n-5.\]
This completes the proof of Proposition \ref{mainprop} and Theorem \ref{mainthm} follows immediately.

\section{Smaller separating sets}
In this section we will show that for $n \geq 5$ there exist smaller separating sets for $\C[\vv^n]^G$ than $S_n$. The aim is prove Proposition \ref{mainprop2}; Theorem \ref{mainthm2} will follow immediately.

We begin by noting that, for $n \geq 4$, $\C[\vv^n]^G$ is not a polynomial ring; for $n \geq 2$ its dimension is $4n-3$ and the minimum number of algebra generators is $\frac16(n^3 + 5n)$. Relations between the generators of $\C[\M^n]^G$ were completely described by Drensky \cite{Drensky2x2}, but we do not need a complete description here. The following is a translation of one kind of relation from \cite{Drensky2x2} to the trace-zero setting:

\begin{equation}\label{tracerel}
t_{ijk}t_{pqr} = -\frac18 \begin{vmatrix} t_{ip} & t_{iq} & t_{ir}\\
							t_{jp} & t_{jq} & t_{jr}\\
							t_{kp} & t_{kq} & t_{kr} \end{vmatrix}
\end{equation} 
for $1 \leq i<j<k \leq n$, $1 \leq p<q<r \leq n$, $n \geq 4$. This can also be shown directly using \eqref{tripletraceasdet} and standard properties of determinants.

For the rest of this section assume $n \geq 3$. Now suppose that $\bA,\bA'$ are trace-zero $n$-matrices and that $\Tr(A_iA_j) = \Tr(A'_iA_j')$ for all $1 \leq i \leq j \leq n$. The relation \eqref{tracerel} above with $i=p,j=q,r=k$ tells us immediately that $\Tr(A_iA_jA_k) = \pm \Tr(A'_iA'_jA'_k)$ for all $1 \leq i<j<k \leq n$. Applying the same relation with arbitrary $i,j,k,p,q,r$ shows that either $\Tr(A_iA_jA_k) = \Tr(A'_iA'_jA'_k)$ for all $1 \leq i<j<k \leq n$ or $\Tr(A_iA_jA_k) = -\Tr(A'_iA_j'A_k')$ for all $1 \leq i<j<k \leq n$. 

Further, suppose $f \in \bigoplus_{ 1 \leq i<j<k \leq n} \C \cdot t_{ijk}$. Then $f(\bA) = \pm f(\bA')$. If $f(\bA) = f(\bA') \neq 0$ then for $1 \leq  i<j<k \leq n$ we have $$t_{ijk}(\bA) = \frac{t_{ijk}f(\bA)}{f(\bA)} =  \frac{t_{ijk}f(\bA')}{f(\bA')} = t_{ijk}(\bA')$$ since $t_{ijk} f \in \C[t_{ij}: 1 \leq i \leq j \leq n]$.
 
From the above discussion we deduce: 
\begin{Lemma}\label{rankdetecting} Let $S' = \{t_{ij}: 1 \leq i \leq j \leq n\} \cup F$ where $F \subset \bigoplus_{ 1 \leq i<j<k \leq n} \C \cdot t_{ijk}$. Then $S'$ is a separating set for $\C[\vv^n]^G$ if and only if for all $\bA \in \vv^n$ we have
\[f(\bA) = 0 \ \text{for all} \ f \in F \Rightarrow t_{ijk}(\bA) = 0 \ \text{for all} \ 1 \leq i < j < k \leq n.\] 
\end{Lemma}

Given $\bA \in \vv^n$ we may form a $3 \times n$ matrix $M_{\bA}$ as follows:
\[M_{\bA} = \begin{pmatrix} c_1 & c_2 & \cdots & c_n \\ b_1 & b_2 & \cdots & b_n \\ a_1 & a_2 & \cdots & a_n  \end{pmatrix}\]
The values of $t_{ijk}(\bA)$ are the maximal minors of this matrix. Let $Y$ be a generic $3 \times n$ matrix and let $I$ be the ideal of $\C[Y]$ generated by its $3 \times 3$ minors. Choosing a set $F$ as in Lemma \ref{rankdetecting} above is equivalent to choosing a set of linear combinations of $3 \times 3$ minors generating an ideal of $\C[Y]$ whose radical is the same as the radical of $I$ (in fact, $I$ is a radical ideal). \cite[Lemma~5.9]{BrunsVetter} describes how to do this, and Proposition 5.20, Corollary 5.21 and Proposition 5.22 [loc. cit.] show that $3n-8$ linear combinations suffice. Concretely, let $[ijk]$ denote the minor corresponding to columns $i,j,k$ in $Y$; choose $e_1, \ldots, e_r$, $r = 3n-8$ where
\[e_l = \sum_{1 \leq i<j<k \leq n} \lambda_{ijk}[ijk]\] for $l=1,\ldots, r$. Set
\[f_l = \sum_{1 \leq i<j<k \leq n} \lambda_{ijk} t_{ijk}\] 
 for $l=1,\ldots, r$. Then 
\[S' = \{t_{ij}: 1 \leq i \leq j \leq n\} \cup \{f_l: l= 1 \ldots, r\}\] is a separating set for $\C[\vv^n]^G$. The cardinality of this set is
\[n+ \dbinom n 2+3n-8 = \frac12(n^2 +7n-16).\] This completes the proof of Proposition \ref{mainprop2}; Theorem \ref{mainthm2} now follows immediately from the remarks at the beginning of Section 3.

\section{Matrix Semi-invariants}

In this section we consider the action of a different but related group on $\M^n_2$: let $H:= \SL_2 \times \SL_2$. This acts on an $n$-matrix $\bA$ according to the formula
\[(h_1,h_2) \cdot \bA = (h_1 A_1  h_2^{-1}, h_1 A_2 h_2^{-1}, \ldots, h_1 A_n h_2^{-1}).\]

Generating sets for the algebras of invariants $\C[\M^n_2]^H$ are known, see \cite{DomokosPoincare}. More recently, Domokos \cite{DomokosSemi} showed that the following set $S_n$ of invariants are a separating set for $\C[\M^n_2]^H$  which is minimal by inclusion (we retain the notation of Section 1):

\begin{itemize}
\item $\det(X_i): 1 \leq i \leq n$;
\item $\langle X_i | X_j \rangle:= \Tr(X_i)\Tr(X_j) - \Tr(X_iX_j): 1 \leq i<j \leq n$;
\item $\xi(X_iX_jX_kX_l): 1 \leq i<j<k<l \leq n.$
\end{itemize}

Here $\xi(X_iX_jX_kX_l)$ is the coefficient of $a_ia_ja_ka_l$ in the determinant
\[\begin{vmatrix} a_iX_i & a_jX_j\\ a_kX_k & a_lX_l \end{vmatrix} \in \C[\M_2^n][a_i,a_j,a_k,a_l].\]

The size of this separating set is $n+ \begin{pmatrix} n\\ 2 \end{pmatrix} + \begin{pmatrix} n\\4 \end{pmatrix}$. We note once again that the fact that this separating set is minimal by inclusion does not mean that it has minimal cardinality.

In this section we prove the following:
\begin{thm}\label{semiinv}
Let $S \subseteq \C[\M^n_2]^H$ be a separating set. Then $|S| \geq 5n-10$.
\end{thm}

The dimension of $\C[\M^n_2]^H$ for $n\geq 3$ is $\dim(\M_2^n) - \dim(H) = 4n-6$. This follows from Proposition \ref{dimensionofinvariantring} because there exist $3$-matrices whose stabiliser in $H$ is the finite group $\pm I$. Contrastingly, $\dim (\C[\M_2^2]^H) = 8-6+1=3$, since every $2$-matrix has at least a 1-dimensional stabiliser, and $\dim(\C[\M_2]^H) = 4-6+3=1$ since the stabiliser of any matrix has dimension at least 3.
 So, our result implies that for $n \geq 5$ there does not exist a polynomial separating set for the action of $H$ on $\M_2^n$. 

The proof is a straightforward application of a result of Domokos: for $n \geq 1$ consider the morphism $\sigma: \M_2^n \rightarrow \M_2^{n+1}$ given by
\[\sigma (A_1,A_2, \ldots, A_n) =  (A_1,A_2, \ldots, A_n,I)\] where $I$ is the $2 \times 2$ identity matrix. By \cite[Proposition~4.1]{DomokosRelative}, the induced morphism
\[\sigma^*: \C[M_2^{n+1}]^H \rightarrow \C[\M_2^n]^G\] is surjective (the reference has $G = \GL_2$ but the algebras of invariants are the same). This can be used to show that for any separating set $S \subseteq \C[M_2^{n+1}]^H$, $\sigma^*(S) \subseteq \C[\M_2^n]^G$ is a separating set, see \cite[Corollary~6.3]{DomokosSemi}. Now Theorem \ref{semiinv} follows immediately from this observation and Theorem \ref{mainthm}.

The table below compares this lower bound with the size of the separating set given in this section.

\begin{center}
\begin{table}[h]
\begin{tabular}{c|cccc}
$n$ & $\dim(\C[\M_2^n]^H)$ & $|S_n| $& Lower bound \\ \hline
%1 & 2 & 2 & \\
2 & 3 & 3 &3 \\
3 & 6 & 6&6\\
4 & 10 & 11& 10\\
5 & 14 & 20& 15\\
6 & 18 & 36 & 20 
\end{tabular}
\end{table}
\end{center}

Note that in the case $n=4$ we cannot rule out the existence of a polynomial separating set with these methods. In a forthcoming paper we plan to compute the separating variety for this action directly and thereby sharpen these lower bounds.

\bibliographystyle{plain}
\bibliography{MyBib}

\begin{thebibliography}{10}

\bibitem{Bedratyuk7}
Leonid Bedratyuk.
\newblock A complete minimal system of covariants for the binary form of degree
  7.
\newblock {\em J. Symbolic Comput.}, 44(2):211--220, 2009.

\bibitem{Belitskii}
Genrich~R. Belitskii and Vladimir~V. Sergeichuk.
\newblock Complexity of matrix problems.
\newblock volume 361, pages 203--222. 2003.
\newblock Ninth Conference of the International Linear Algebra Society (Haifa,
  2001).

\bibitem{BrunsVetter}
Winfried Bruns and Udo Vetter.
\newblock {\em Determinantal rings}, volume~45 of {\em Monograf\'{\i}as de
  Matem\'{a}tica [Mathematical Monographs]}.
\newblock Instituto de Matem\'{a}tica Pura e Aplicada (IMPA), Rio de Janeiro,
  1988.

\bibitem{DerksenKemper}
Harm Derksen and Gregor Kemper.
\newblock {\em Computational invariant theory}.
\newblock Invariant Theory and Algebraic Transformation Groups, I.
  Springer-Verlag, Berlin, 2002.
\newblock Encyclopaedia of Mathematical Sciences, 130.

\bibitem{Dolgachev}
Igor Dolgachev.
\newblock {\em Lectures on invariant theory}, volume 296 of {\em London
  Mathematical Society Lecture Note Series}.
\newblock Cambridge University Press, Cambridge, 2003.

\bibitem{DomokosSemi}
M.~Domokos.
\newblock Characteristic free description of semi-invariants of {$2\times2$}
  matrices.
\newblock {\em J. Pure Appl. Algebra}, 224(5):106220, 13, 2020.

\bibitem{DomokosPoincare}
M\'{a}ty\'{a}s Domokos.
\newblock Poincar\'{e} series of semi-invariants of {$2\times 2$} matrices.
\newblock {\em Linear Algebra Appl.}, 310(1-3):183--194, 2000.

\bibitem{DomokosRelative}
M\'{a}ty\'{a}s Domokos.
\newblock Relative invariants of {$3\times 3$} matrix triples.
\newblock {\em Linear and Multilinear Algebra}, 47(2):175--190, 2000.

\bibitem{Drensky2x2}
Vesselin Drensky.
\newblock Defining relations for the algebra of invariants of {$2\times 2$}
  matrices.
\newblock {\em Algebr. Represent. Theory}, 6(2):193--214, 2003.

\bibitem{DrezetLuna}
Jean-Marc Dr\'{e}zet.
\newblock Luna's slice theorem and applications.
\newblock In {\em Algebraic group actions and quotients}, pages 39--89. Hindawi
  Publ. Corp., Cairo, 2004.

\bibitem{DufresnePhD}
Emilie Dufresne.
\newblock Separating invariants.
\newblock {\em Ph. D. thesis, Kingston Ontario}, 2008.

\bibitem{DufresneSeparating}
Emilie Dufresne.
\newblock Separating invariants and finite reflection groups.
\newblock {\em Adv. Math.}, 221:1979--1989, 2009.

\bibitem{DufresneElmerSezer}
Emilie Dufresne, Jonathan Elmer, and M\"{u}fit Sezer.
\newblock Separating invariants for arbitrary linear actions of the additive
  group.
\newblock {\em Manuscripta Math.}, 143(1-2):207--219, 2014.

\bibitem{DufresneJeffries}
Emilie Dufresne and Jack Jeffries.
\newblock Separating invariants and local cohomology.
\newblock {\em Adv. Math.}, 270:565--581, 2015.

\bibitem{DufresneJeffriesTori}
Emilie Dufresne and Jack Jeffries.
\newblock Mapping toric varieties into low dimensional spaces.
\newblock {\em Trans. Amer. Math. Soc.}, to appear.

\bibitem{DufresneKohlsFiniteSep}
Emilie Dufresne and Martin Kohls.
\newblock A finite separating set for {D}aigle and {F}reudenburg's
  counterexample to {H}ilbert's fourteenth problem.
\newblock {\em Comm. Algebra}, 38(11):3987--3992, 2010.

\bibitem{DufresneKohlsSepVar}
Emilie Dufresne and Martin Kohls.
\newblock The separating variety for the basic representations of the additive
  group.
\newblock {\em J. Algebra}, 377:269--280, 2013.

\bibitem{DufresneKraft}
Emilie Dufresne and Hanspeter Kraft.
\newblock Invariants and separating morphisms for algebraic group actions.
\newblock {\em Math. Z.}, 280(1-2):231--255, 2015.

\bibitem{Eisenbud}
David Eisenbud.
\newblock {\em Commutative algebra with a view toward algebraic geometry},
  volume 150 of {\em Graduate Texts in Mathematics}.
\newblock Springer-Verlag, New York, 1995.

\bibitem{ElmerKohls}
Jonathan Elmer and Martin Kohls.
\newblock Separating invariants for the basic {$\Bbb G_{a}$}-actions.
\newblock {\em Proc. Amer. Math. Soc.}, 140(1):135--146, 2012.

\bibitem{FleischmannNoetherBound}
Peter Fleischmann.
\newblock The {N}oether bound in invariant theory of finite groups.
\newblock {\em Adv. Math.}, 156(1):23--32, 2000.

\bibitem{florentino2x2}
Carlos A.~A. Florentino.
\newblock Simultaneous similarity and triangularization of sets of 2 by 2
  matrices.
\newblock {\em Linear Algebra Appl.}, 431(9):1652--1674, 2009.

\bibitem{FogartyNoetherBound}
John Fogarty.
\newblock On {N}oether's bound for polynomial invariants of a finite group.
\newblock {\em Electron. Res. Announc. Amer. Math. Soc.}, 7:5--7 (electronic),
  2001.

\bibitem{Giral}
Jos\'{e}~M. Giral.
\newblock Krull dimension, transcendence degree and subalgebras of finitely
  generated algebras.
\newblock {\em Arch. Math. (Basel)}, 36(4):305--312, 1981.

\bibitem{SGA2}
Alexander Grothendieck.
\newblock {\em Cohomologie locale des faisceaux coh\'{e}rents et
  th\'{e}or\`emes de {L}efschetz locaux et globaux {$(SGA$} {$2)$}}.
\newblock North-Holland Publishing Co., Amsterdam; Masson \& Cie, \'{E}diteur,
  Paris, 1968.
\newblock Augment\'{e} d'un expos\'{e} par Mich\`ele Raynaud, S\'{e}minaire de
  G\'{e}om\'{e}trie Alg\'{e}brique du Bois-Marie, 1962, Advanced Studies in
  Pure Mathematics, Vol. 2.

\bibitem{KamkeKemper}
Tobias Kamke and Gregor Kemper.
\newblock Algorithmic invariant theory of nonreductive groups.
\newblock {\em Qual. Theory Dyn. Syst.}, 11(1):79--110, 2012.

\bibitem{Lopatin2x2}
Ivan Kaygorodov, Artem Lopatin, and Yury Popov.
\newblock Separating invariants for {$2\times2$} matrices.
\newblock {\em Linear Algebra Appl.}, 559:114--124, 2018.

\bibitem{KemperCompRed}
Gregor Kemper.
\newblock Computing invariants of reductive groups in positive characteristic.
\newblock {\em Transform. Groups}, 8(2):159--176, 2003.

\bibitem{KemperSeparating}
Gregor Kemper.
\newblock Separating invariants.
\newblock {\em Journal of Symbolic Computation}, 44(9):1212 -- 1222, 2009.
\newblock Effective Methods in Algebraic Geometry.

\bibitem{LeBruynProcesi}
Lieven Le~Bruyn and Claudio Procesi.
\newblock Semisimple representations of quivers.
\newblock {\em Trans. Amer. Math. Soc.}, 317(2):585--598, 1990.

\bibitem{Mumford}
D.~Mumford, J.~Fogarty, and F.~Kirwan.
\newblock {\em Geometric invariant theory}, volume~34 of {\em Ergebnisse der
  Mathematik und ihrer Grenzgebiete (2)}.
\newblock Springer-Verlag, Berlin, third edition, 1994.

\bibitem{reimers}
Fabian Reimers.
\newblock Separating invariants of finite groups.
\newblock {\em J. Algebra}, 507:19--46, 2018.

\bibitem{SezerCyclic}
M{\"u}fit Sezer.
\newblock Explicit separating invariants for cyclic $p$-groups.
\newblock {\em J. Combin. Theory Ser. A}, (doi:10.1016/j/jcta.2010.05.003),
  2010.

\bibitem{PopovVinberg}
\`E.~B. Vinberg and V.~L. Popov.
\newblock Invariant theory.
\newblock In {\em Algebraic geometry, 4 ({R}ussian)}, Itogi Nauki i Tekhniki,
  pages 137--314, 315. Akad. Nauk SSSR, Vsesoyuz. Inst. Nauchn. i Tekhn.
  Inform., Moscow, 1989.

\bibitem{WehlauCyclicViaClassical}
David~L. Wehlau.
\newblock Invariants for the modular cyclic group of prime order via classical
  invariant theory.
\newblock {\em J. Eur. Math. Soc. (JEMS)}, 15(3):775--803, 2013.

\end{thebibliography}

\end{document}